\documentclass[oneside,english]{amsart}
\usepackage[T1]{fontenc}
\usepackage[latin9]{inputenc}
\usepackage{amsthm}
\usepackage{amssymb}
\usepackage{esint}

\makeatletter
\numberwithin{equation}{section}
\numberwithin{figure}{section}
\theoremstyle{plain}
\newtheorem{thm}{\protect\theoremname}
  \theoremstyle{remark}
  \newtheorem{rem}[thm]{\protect\remarkname}
  \theoremstyle{definition}
  \newtheorem{defn}[thm]{\protect\definitionname}
  \theoremstyle{plain}
  \newtheorem{fact}[thm]{\protect\factname}
  \theoremstyle{plain}
  \newtheorem{lem}[thm]{\protect\lemmaname}

\makeatother

\usepackage{babel}
  \providecommand{\definitionname}{Definition}
  \providecommand{\factname}{Fact}
  \providecommand{\lemmaname}{Lemma}
  \providecommand{\remarkname}{Remark}
\providecommand{\theoremname}{Theorem}

\begin{document}

\title{Uniform Factorial Decay Estimate for the Remainder of Rough Taylor
Expansion}

\author{Horatio Boedihardjo, Terry Lyons, Danyu Yang}

\address{Oxford-Man Institute of Quantitative Finance, Eagle House, Walton
Well Road, Oxford. OX2 6ED, UK. }
\begin{abstract}
We establish an uniform factorial decay estimate for the Taylor approximation
of solutions to controlled differential equations. Its proof requires
a factorial decay estimate for controlled paths which is interesting
in its own right. 
\end{abstract}
\maketitle

\section{ntroduction }

For a controlled differential equation of the form 
\begin{eqnarray}
\mathrm{d}Y_{t} & = & f\left(Y_{t}\right)\mathrm{d}X_{t}\nonumber \\
Y_{0} & = & y_{0}.\label{eq:controlled differential equation}
\end{eqnarray}
where $X:\left[0,T\right]\rightarrow\mathbb{R}^{d}$ is a path with
bounded variation and $f:\mathbb{R}^{e}\rightarrow L\left(\mathbb{R}^{d},\mathbb{R}^{e}\right)$
is a smooth vector field, we are interested in estimating 
\begin{eqnarray}
 &  & Y_{t}-Y_{s}-\sum_{k=1}^{N}f^{\circ k}\left(Y_{s}\right)\int_{s<s_{1}<\ldots<s_{k}<t}\mathrm{d}X_{s_{1}}\otimes\ldots\otimes\mathrm{d}X_{s_{k}}\label{eq:ode taylor}\\
 & \equiv & \int_{s<s_{1}<\ldots<s_{N}<t}f^{\circ N}\left(Y_{s_{1}}\right)-f^{\circ N}\left(Y_{s}\right)\mathrm{d}X_{s_{1}}\otimes\ldots\otimes\mathrm{d}X_{s_{N}},
\end{eqnarray}
where $f^{\circ m}:\mathbb{R}^{e}\rightarrow L\left(\left(\mathbb{R}^{d}\right)^{\otimes m},\mathbb{R}^{e}\right)$
is defined inductively by 
\begin{eqnarray*}
f^{\circ1} & = & f\\
f^{\circ k+1} & = & D\left(f^{\circ k}\right)f.
\end{eqnarray*}
The iterated integrals in (\ref{eq:ode taylor}) will appear numerous
times and we shall use the shorthand
\begin{equation}
X_{s,t}^{k}:=\int_{s<s_{1}<\ldots<s_{k}<t}\mathrm{d}X_{s_{1}}\otimes\ldots\otimes\mathrm{d}X_{s_{k}}.\label{eq:iterated integrals}
\end{equation}

For $p=1$, since the $1-$variation norm of $X$ equals to the $L^{1}$
norm of the derivative of $X$, we have (see for example \cite{FV08})
\begin{equation}
\left|Y_{t}-Y_{s}-\sum_{k=1}^{N}f^{\circ k}\left(Y_{s}\right)X_{s,t}^{k}\right|\leq\left\Vert f^{\circ N}\right\Vert _{\infty}\left\Vert Df\right\Vert _{\infty}\frac{\left|X\right|_{1-var;\left[s,t\right]}^{N+1}}{N!}\label{eq:ODe remainder}
\end{equation}
where 
\[
\left|X\right|_{1-var;\left[s,t\right]}=\sup_{s<t_{1}<\ldots<t_{n}<t}\sum_{i=0}^{n}\left|X_{t_{i+1}}-X_{t_{i}}\right|
\]
and $\left\Vert f^{\circ N}\right\Vert _{\infty}$ denote $\sup_{x\in\mathbb{R}^{e}}\left|f^{\circ N}\left(x\right)\right|$
with $\left|\cdot\right|$ denoting the operator norm 
\[
\left|f^{\circ N}\left(x\right)\right|=\sup_{v\in\left(\mathbb{R}^{d}\right)^{\otimes N}}\frac{\left|f^{\circ N}\left(x\right)\left(v\right)\right|}{\left\Vert v\right\Vert }.
\]

The estimate (\ref{eq:ODe remainder}), when the $1$-variation metric
is replaced by the $p$-variation metric, has been shown in \cite{Dav07}
($p<3$), \cite{Gub04} ($p<3$) and \cite{FV08} (all $p\geq1$)
without the factorial decay factor. We shall prove that
\begin{thm}
\label{thm:main result}Let $X=\left(1,X^{1},\ldots,X^{\lfloor p\rfloor}\right)$
be a $p$-weak geometric rough path. Let $f$ be a Lip($\gamma$)
vector field where $\gamma>p-1$. Let $Y$ be a solution to the rough
differential equation 
\begin{equation}
\mathrm{d}Y_{t}=f\left(Y_{t}\right)\mathrm{d}X_{t}\label{eq:controlled differential equation 2}
\end{equation}
defined in the sense of \cite{FV10}. Then there exists a constant
$C_{p}$ depending only on $p$ such that 
\begin{equation}
\left|Y_{t}-Y_{s}-\sum_{k=1}^{\lfloor\gamma\rfloor}f^{\circ k}\left(Y_{s}\right)X_{s,t}^{k}\right|\leq\frac{1}{\left(\frac{\lfloor\gamma\rfloor}{p}\right)!}\beta^{\lfloor\gamma\rfloor}M_{p,\gamma}\max_{\lfloor\gamma\rfloor-\lfloor p\rfloor+1\leq m\leq\lfloor\gamma\rfloor}\left|f^{\circ m}\right|_{Lip\left(1\right)}\omega\left(s,t\right)^{\frac{\gamma}{p}},\label{eq:main inequality}
\end{equation}
where 
\begin{eqnarray}
M_{p,\gamma} & = & 2C_{p}\left(\left|f\right|_{Lip\left(\gamma\wedge\lfloor p\rfloor+1\right)}\vee1\right)^{\lfloor p\rfloor+1}\left(\left|X\right|_{p-var,\left[0,T\right]}\vee1\right)^{p+1}\nonumber \\
\beta & = & p\left(1+\sum_{r=2}^{\infty}\left(\frac{2}{r-1}\wedge1\right)^{\frac{\lfloor p\rfloor+1}{p}}\right).\label{eq:beta condition}
\end{eqnarray}

\end{thm}
We refer the readers to Definition 9.16 and Definition 10.2 in \cite{FV10}
for the definition of Lip ($\gamma$) vector fields and weak geometric
rough paths respectively. We shall however recall the definition of
$p$-variation and some basic notations in Section 2.
\begin{rem}
If the equation (\ref{eq:controlled differential equation 2}) has
more than one solution, then any solution must satisfy (\ref{eq:main inequality}). 
\end{rem}

\begin{rem}
Taking the biggest $\gamma$ may not give the best estimate in Theorem
\ref{thm:main result}. In general the term $\max_{\lfloor\gamma\rfloor-\lfloor p\rfloor+1\leq m\leq\lfloor\gamma\rfloor}\left|f^{\circ m}\right|_{Lip\left(1\right)}$
could grow factorially fast in $\gamma$. Since a Lip($\gamma$) function
is also Lip($\gamma^{\prime}$) for all $\gamma^{\prime}<\gamma$,
we may choose $\gamma^{\prime}$ which optimises the estimate (\ref{eq:main inequality}).
\end{rem}
The proof for (\ref{eq:ODe remainder}) relies heavily on the relation
between the $1$-variation of the path and the $L^{1}$ norm of its
derivative. Proving an estimate of the form (\ref{eq:ODe remainder})
for the $p$-variation metric, even without the factorial factor,
requires the clever idea of Young\cite{Young}. The integration with
respect to a path can be expressed in terms of the limit of a Riemann
sums as the size of partition converges to zero. Young's idea was
to estimate the Riemann sum with respect to a partition by removing
points from the partition successively. This idea had been used in
\cite{Lyons94} to show that, for $p<2$, the path iterated integrals
of order $n$ decays at the speed of 
\begin{equation}
c_{p}^{n}\left(\frac{1}{n!}\right)^{\frac{1}{p}}\left\Vert X\right\Vert _{p-var,\left[s,t\right]}^{n}.\label{eq:decay rate}
\end{equation}
with an explicit function $c_{p}$ depending only on $p$ but not
on $n$ nor the path. T. Lyons' proof for the $p\geq2$ case in \cite{Lyons98}
is slightly different and used the neoclassical inequality (\cite{Lyons98},\cite{HaraHino10})
\[
\sum_{k=0}^{N}\frac{1}{\Gamma\left(k/p+1\right)\Gamma\left(\left(n-k\right)/p+1\right)}a^{k/p}b^{\left(n-k\right)/p}\leq p\frac{1}{\Gamma\left(n/p+1\right)}\left(a+b\right)^{n/p}
\]
to obtain a decay rate of the form 
\[
c_{p}^{n}\frac{1}{\Gamma\left(n/p+1\right)}\left\Vert X\right\Vert _{p-var,\left[s,t\right]}^{n}
\]
 where $\Gamma$ is the Gamma function. In \cite{Boe15}, the factorial
decay for the iterated integrals of Branched rough paths had been
established through extending Lyons' earlier technique in \cite{Lyons94}
to the $p\geq2$ regime. In particular, this provides an alternative
proof for the decay of iterated integrals in the $p\geq2$ case without
the use of the neoclassical inequality. In this paper, the fact that
the ``$N$'' in (\ref{eq:ODe remainder}) is greater than $\lfloor p\rfloor$
forced us to use the approach in \cite{Lyons98} instead of that of
\cite{Lyons94}. New ideas will be required to extend our main result
to Branched rough paths as neoclassical inequality does not hold when
the factorial is replaced by factorial for rooted trees.

\thanks{The authors are grateful for the support of the ERC Advanced grant
(grant agreement no. 291244), for which the second author is the principal
investigator. We would also like to thank H. Oberhauser for the useful
discussions.}

\section{The Proof}

\subsection{Notations and basic definitions}

For each $k\in\mathbb{N}$, we equip a norm on $\left(\mathbb{R}^{d}\right)^{\otimes k}$
by identifying it with $\mathbb{R}^{d^{k}}$. If $\pi_{k}$ denotes
the projection $1\oplus\mathbb{R}^{d}\oplus\ldots\oplus\left(\mathbb{R}^{d}\right)^{N}\rightarrow\left(\mathbb{R}^{d}\right)^{\otimes k}$,
then we define a norm on $1\oplus\mathbb{R}^{d}\oplus\ldots\oplus\left(\mathbb{R}^{d}\right)^{N}$
by 
\[
\left\Vert x\right\Vert =\max_{1\leq k\leq N}\left\Vert \pi_{k}\left(x\right)\right\Vert ^{\frac{1}{k}}.
\]

\begin{defn}
Let $T>0$ and $p\geq1$. A path $X:\left[0,T\right]\rightarrow1\oplus\mathbb{R}^{d}\oplus\ldots\oplus\left(\mathbb{R}^{d}\right)^{\lfloor p\rfloor}$
has finite $p$-variation if for all $0<s<t<T$, 
\begin{equation}
\left\Vert X\right\Vert _{p-var,\left[s,t\right]}:=\sup_{s<t_{1}<\ldots<t_{n}<t}\max_{1\leq k\leq\lfloor p\rfloor}\left(\sum_{i=1}^{n}\left\Vert \pi_{k}\left(X_{s}^{-1}X_{t}\right)\right\Vert ^{\frac{p}{k}}\right)^{\frac{1}{p}}<\infty\label{eq:p-variation}
\end{equation}
where $X^{-1}$ denote the unique multiplicative inverse of $X\in1\oplus\mathbb{R}^{d}\oplus\ldots\oplus\left(\mathbb{R}^{d}\right)^{\lfloor p\rfloor}$
\end{defn}
We first recall Lyons' extension theorem, which will be used multiple
times in what follows, and in the following form: 
\begin{fact}
\label{Lyons extension}(Theorem 2.2.1 in \cite{Lyons98}) Let $p\geq1$
and $X=\left(1,X^{1},\ldots,X^{\lfloor p\rfloor}\right)$ be a $p$-weak
geometric rough path. Then there exists a unique continuous path $\mathbf{X}=\left(1,X^{1},\ldots\right)\in T\left(\left(\mathbb{R}^{d}\right)\right)$
which extends $X$, $\mathbf{X}_{0}=\left(1,0\ldots\right)$ and for
all $l\geq\lfloor p\rfloor$, 
\[
\sup_{s<t_{1}<\ldots<t_{n}<t}\left(\sum_{i=1}^{n}\left\Vert \pi_{l}\left(\mathbf{X}_{t_{i}}^{-1}\mathbf{X}_{t_{i+1}}\right)\right\Vert ^{\frac{p}{l}}\right)^{\frac{1}{p}}\leq\frac{\beta^{l-1}}{\left(\frac{l}{p}\right)!}\left\Vert X\right\Vert _{p-var,\left[s,t\right]}^{l}.
\]
\end{fact}
\begin{rem}
Note that for paths with finite $1$-variation, the $\left(X^{k}\right)_{k\geq1}$
defined in this Theorem are exactly the iterated integrals of $X$.
Hence no confusion will arise by using this same notation as in (\ref{eq:iterated integrals}).
\end{rem}

\subsection{The Proof}
\begin{lem}
\label{lem:control path lemma}Let $p\geq1$ and $\gamma>p-1$. Let
$\left(1,X^{1},\ldots,X^{\lfloor p\rfloor}\right)$ be a $p$-weak
geometric rough path. Let $Y^{\left(i\right)}$ be a function $\left[0,T\right]\rightarrow L\left(\left(\mathbb{R}^{d}\right)^{\otimes i},\mathbb{R}^{e}\right)$
and $\left(Y^{\left(0\right)},Y^{\left(1\right)},\ldots,Y^{\left(\lfloor\gamma\rfloor\right)}\right)$
satisfies, for $\lfloor\gamma\rfloor-\lfloor p\rfloor+1\leq m\leq\lfloor\gamma\rfloor$,
\begin{equation}
\left|Y_{s}^{\left(m\right)}-\sum_{l=0}^{\lfloor\gamma\rfloor-m}Y_{s}^{\left(l+m\right)}X_{s,t}^{l}\right|\leq\frac{1}{\left(\frac{\lfloor\gamma\rfloor-m+1}{p}\right)!}M\beta^{\lfloor\gamma\rfloor-m}\left\Vert X\right\Vert _{p-var,\left[s,t\right]}^{\gamma-m},\label{eq:first p control}
\end{equation}
and for $m\leq\lfloor\gamma\rfloor-\lfloor p\rfloor$, 
\begin{equation}
\int_{s}^{t}Y_{s_{1}}^{\left(m+1\right)}\mathrm{d}X_{s_{1}}=Y_{t}^{\left(m\right)}-Y_{s}^{\left(m\right)}.\label{eq:rpFTC}
\end{equation}
For $l\geq\lfloor p\rfloor+1$, let $X^{l}$ denote the projection
to $\left(\mathbb{R}^{d}\right)^{\otimes l}$ of the unique extension
of $\left(1,X^{1},\ldots,X^{\lfloor p\rfloor}\right)$ given in Fact
\ref{Lyons extension}. Then (\ref{eq:first p control}) holds for
all $0\leq m\leq\lfloor\gamma\rfloor$. \end{lem}
\begin{proof}
We shall carry out backward induction on $k$ starting from $\lfloor\gamma\rfloor-\lfloor p\rfloor$
and moving all the way down to $1$. 

The base induction step holds because of the assumption. We will assume
from now that $k\leq\lfloor\gamma\rfloor-\lfloor p\rfloor$. 

For the induction step, note that by (\ref{eq:rpFTC}) and that $\lfloor\gamma\rfloor-k\geq\lfloor p\rfloor$,
if $\mathcal{P}=\left(t_{0}<t_{1}<\ldots<t_{n}\right)$, then 
\begin{equation}
Y_{t}^{\left(k\right)}-\sum_{l=0}^{\lfloor\gamma\rfloor-k}Y_{s}^{\left(k+l\right)}X_{s,t}^{l}=\lim_{\left|\mathcal{P}\right|\rightarrow0}\sum_{i=0}^{\left|\mathcal{P}\right|}\sum_{l=1}^{\lfloor\gamma\rfloor-k}\left(Y_{t_{i}}^{\left(k+l\right)}-\sum_{l_{1}=0}^{\lfloor\gamma\rfloor-k-l}Y_{s}^{\left(k+l+l_{1}\right)}X_{s,t_{i}}^{l_{1}}\right)X_{t_{i},t_{i+1}}^{l}.\label{eq:inductive}
\end{equation}

We first show that the term
\begin{equation}
\sum_{i=0}^{\left|\mathcal{P}\right|}\sum_{l=1}^{\lfloor\gamma\rfloor-k}\sum_{l_{1}=0}^{\lfloor\gamma\rfloor-k-l}Y_{s}^{\left(k+l+l_{1}\right)}X_{s,t_{i}}^{l_{1}}X_{t_{i},t_{i+1}}^{l}.\label{eq:indep of partition}
\end{equation}
 is in fact independent of $\mathcal{P}$. 
\begin{eqnarray*}
 &  & \sum_{i=0}^{\left|\mathcal{P}\right|}\sum_{l=1}^{\lfloor\gamma\rfloor-k}\sum_{l_{1}=0}^{\lfloor\gamma\rfloor-k-l}Y_{s}^{\left(k+l+l_{1}\right)}X_{s,t_{i}}^{l_{1}}X_{t_{i},t_{i+1}}^{l}\\
 & = & \sum_{i=0}^{\left|\mathcal{P}\right|}\left[\sum_{0\leq l+l_{1}\leq\lfloor\gamma\rfloor-k}Y_{s}^{\left(k+l+l_{1}\right)}X_{s,t_{i}}^{l_{1}}X_{t_{i},t_{i+1}}^{l}-\sum_{l_{1}=0}^{\lfloor\gamma\rfloor-k}Y_{s}^{\left(k+l_{1}\right)}X_{s,t_{i}}^{l_{1}}\right]\\
 & = & \sum_{i=0}^{\left|\mathcal{P}\right|}\left[\sum_{r=0}^{\lfloor\gamma\rfloor-k}\sum_{l+l_{1}=r}Y_{s}^{\left(k+r\right)}X_{s,t_{i}}^{l_{1}}X_{t_{i},t_{i+1}}^{l}-\sum_{l_{1}=0}^{\lfloor\gamma\rfloor-k}Y_{s}^{\left(k+l_{1}\right)}X_{s,t_{i}}^{l_{1}}\right]\\
 & = & \sum_{i=0}^{\left|\mathcal{P}\right|}\left[\sum_{r=0}^{\lfloor\gamma\rfloor-k}Y_{s}^{\left(k+r\right)}X_{s,t_{i+1}}^{r}-\sum_{r=0}^{\lfloor\gamma\rfloor-k}Y_{s}^{\left(k+r\right)}X_{s,t_{i}}^{r}\right]\\
 & = & \sum_{r=1}^{\lfloor\gamma\rfloor-k}Y_{s}^{\left(k+r\right)}X_{s,t}^{r}.
\end{eqnarray*}
Let 
\[
\left(Y_{s}^{\left(k\right)}-\sum_{l=0}^{\lfloor\gamma\rfloor-k}Y_{s}^{\left(l\right)}X_{s,t}^{l}\right)^{\mathcal{P}}=\sum_{i=0}^{\left|\mathcal{P}\right|}\sum_{l=1}^{\lfloor\gamma\rfloor-k}\left(Y_{t_{i}}^{\left(k+l\right)}-\sum_{l_{1}=0}^{\lfloor\gamma\rfloor-k-l}Y_{s}^{\left(k+l+l_{1}\right)}X_{s,t_{i}}^{l_{1}}\right)X_{t_{i},t_{i+1}}^{l}.
\]
Since (\ref{eq:indep of partition}) is independent of the partition,
\begin{eqnarray}
 &  & \left(Y_{s}^{\left(k\right)}-\sum_{l=0}^{\lfloor\gamma\rfloor-k}Y_{s}^{\left(l\right)}X_{s,t}^{l}\right)^{\mathcal{P}}-\left(Y_{s}^{\left(k\right)}-\sum_{l=0}^{\lfloor\gamma\rfloor-k}Y_{s}^{\left(l\right)}X_{s,t}^{l}\right)^{\mathcal{P}\backslash\left\{ t_{j}\right\} }\\
 &  & =\sum_{l=1}^{\lfloor\gamma\rfloor-k}Y_{t_{j-1}}^{\left(k+l\right)}X_{t_{j-1},t_{j}}^{l}+\sum_{l=1}^{\lfloor\gamma\rfloor-k}Y_{t_{j}}^{\left(k+l\right)}X_{t_{j},t_{j+1}}^{l}-\sum_{l=1}^{\lfloor\gamma\rfloor-k}Y_{t_{j-1}}^{\left(k+l\right)}X_{t_{j-1},t_{j+1}}^{l}\nonumber \\
 &  & =\sum_{l=1}^{\lfloor\gamma\rfloor-k}\left(Y_{t_{j}}^{\left(k+l\right)}-\sum_{l_{1}=0}^{\lfloor\gamma\rfloor-k-l}Y_{t_{j-1}}^{\left(k+l+l_{1}\right)}X_{t_{j-1},t_{j}}^{l_{1}}\right)X_{t_{j},t_{j+1}}^{l}.\label{eq:after subtraction}
\end{eqnarray}

By induction hypothesis (\ref{eq:first p control}) which holds for
$m>k$ and Theorem 2.2.1 in \cite{Lyons98}, 
\begin{eqnarray}
 &  & \left|\sum_{l=1}^{\lfloor\gamma\rfloor-k}\left(Y_{t_{j}}^{\left(k+l\right)}-\sum_{l_{1}=0}^{\lfloor\gamma\rfloor-k-l}Y_{t_{j-1}}^{\left(k+l+l_{1}\right)}X_{t_{j-1},t_{j}}^{l_{1}}\right)X_{t_{j},t_{j+1}}^{l}\right|\nonumber \\
 & \leq & \sum_{l=1}^{\lfloor\gamma\rfloor-k}\frac{1}{\left(\frac{\lfloor\gamma\rfloor-k-l}{p}!\right)\left(\frac{l}{p}!\right)}M\beta^{\lfloor\gamma\rfloor-k-l}\left\Vert X\right\Vert _{p-var,\left[t_{j-1},t_{j}\right]}^{\gamma-k-l}\beta^{l-1}\left\Vert X\right\Vert _{p-var,\left[t_{j},t_{j+1}\right]}^{l}\nonumber \\
 & \leq & \frac{1}{\left(\frac{\lfloor\gamma\rfloor-k}{p}!\right)}\frac{p}{\beta}M\beta^{\lfloor\gamma\rfloor-k}\left\Vert X\right\Vert _{p-var,\left[t_{j-1},t_{j+1}\right]}^{\gamma-k},\label{eq:drop points}
\end{eqnarray}
where the final line is obtained by the neoclassical inequality in
\cite{HaraHino10}. 

Let $\omega\left(s,t\right)=\left\Vert X\right\Vert _{p-var,\left[s,t\right]}^{p}$.
We now choose $j$ such that, for $\left|\mathcal{P}\right|\geq2$,
\[
\omega\left(t_{j-1},t_{j+1}\right)\leq\left(\frac{2}{\left|\mathcal{P}\right|-1}\wedge1\right)\omega\left(s,t\right)
\]
which exists since 
\[
\sum_{i=1}^{\left|\mathcal{P}\right|-1}\omega\left(t_{i-1},t_{i+1}\right)\leq2\omega\left(s,t\right)
\]
and also that 
\[
\omega\left(t_{j-1},t_{j+1}\right)\leq\omega\left(s,t\right)
\]
 for all $j$. Then as $\gamma-k\geq\lfloor p\rfloor+1$, (\ref{eq:drop points})
is less than or equal to 
\[
\frac{1}{\left(\frac{\lfloor\gamma\rfloor-k}{p}!\right)}\frac{p}{\beta}M\beta^{\lfloor\gamma\rfloor-k}\left(\frac{2}{\left|\mathcal{P}\right|-1}\wedge1\right)^{\frac{\lfloor p\rfloor+1}{p}}\left\Vert X\right\Vert _{p-var,\left[s,t\right]}^{\gamma-k}.
\]

By removing points successively from $\mathcal{P}$ and using that
$\left(Y_{s}^{\left(k\right)}-\sum_{l=0}^{N-k}Y_{s}^{\left(k+l\right)}X_{s,t}^{l}\right)^{\left\{ s,t\right\} }=0$,
we have 
\begin{eqnarray*}
\left|\left(Y_{s}^{\left(k\right)}-\sum_{l=0}^{N-k}Y_{s}^{\left(l\right)}X_{s,t}^{l}\right)^{\mathcal{P}}\right| & \leq & \frac{1}{\left(\frac{\lfloor\gamma\rfloor-k}{p}!\right)}\frac{p}{\beta}M\beta^{\lfloor\gamma\rfloor-k}\sum_{r=2}^{\infty}\left(\frac{2}{r-1}\wedge1\right)^{\frac{\lfloor p\rfloor+1}{p}}\left\Vert X\right\Vert _{p-var,\left[s,t\right]}^{\gamma-k}\\
 & \leq & \frac{1}{\left(\frac{\lfloor\gamma\rfloor-k}{p}!\right)}M\beta^{\lfloor\gamma\rfloor-k}\sum_{r=2}^{\infty}\left(\frac{2}{r-1}\wedge1\right)^{\frac{\lfloor p\rfloor+1}{p}}\left\Vert X\right\Vert _{p-var,\left[s,t\right]}^{\gamma-k},
\end{eqnarray*}
where the final line follows from (\ref{eq:beta condition}). 

By taking limit as $\left|\mathcal{P}\right|\rightarrow0$, (\ref{eq:first p control})
follows for $m=k$. 
\end{proof}
\begin{proof}[Proof of Theorem 1]The only thing to prove is that
$\left(Y,f^{\circ1}\left(Y\right),\ldots,f^{\circ\left(\lfloor\gamma\rfloor\right)}\left(Y\right)\right)$
satisfies Lemma \ref{lem:control path lemma}. 

Let $x^{s,t}:\left[s,t\right]\rightarrow\mathbb{R}^{d}$ be a continuous
path with finite $1$-variation such that 
\[
S_{\lfloor p\rfloor}\left(x^{s,t}\right)_{s,t}=S_{\lfloor p\rfloor}\left(X\right)_{s,t},
\]
and 
\[
\int_{s}^{t}\left|\mathrm{d}x_{u}^{s,t}\right|\leq c_{p}\left|X\right|_{p-var,\left[s,t\right]}
\]
for a function $c_{p}$ is $p$ which is specified in \cite{FV10}
along with the existence of $x^{s,t}$. 

Let $\pi\left(s,Y_{s};x^{s,t}\right)$ denote the solution to the
rough differential equation 
\begin{eqnarray*}
\mathrm{d}Y_{t}^{s,t} & = & f\left(Y_{t}^{s,t}\right)\mathrm{d}x^{s,t}\\
Y_{s}^{s,t} & = & Y_{s}.
\end{eqnarray*}

Note that for $\lfloor\gamma\rfloor-\lfloor p\rfloor+1\leq m\leq\lfloor\gamma\rfloor$, 

\begin{eqnarray}
 &  & \left|f^{\circ\left(m\right)}\left(Y_{t}\right)-\sum_{k=0}^{\lfloor\gamma\rfloor-m}f^{\circ\left(m+k\right)}\left(Y_{s}\right)X_{s,t}^{k}\right|\nonumber \\
 & \leq & \left|f^{\circ m}\left(Y_{t}\right)-f^{\circ m}\left(Y_{t}^{s,t}\right)\right|+\left|f^{\circ m}\left(Y_{t}^{s,t}\right)-\sum_{k=0}^{\lfloor\gamma\rfloor-m}f^{\circ\left(m+k\right)}\left(Y_{s}\right)S_{k}\left(x^{s,t}\right)_{s,t}\right|\label{eq:approx by smooth paths}
\end{eqnarray}

By Theorem 10.16 in \cite{FV10}, 
\begin{eqnarray*}
\left|Y_{t}-Y_{t}^{s,t}\right| & \leq & C_{p}\left|f\right|_{Lip\left(\gamma\wedge\left(\lfloor p\rfloor+1\right)\right)}^{\gamma\wedge\left(\lfloor p\rfloor+1\right)}\left|X\right|_{p-var,\left[s,t\right]}^{\gamma\wedge\left(\lfloor p\rfloor+1\right)}.
\end{eqnarray*}

Therefore, 
\begin{eqnarray}
 &  & \left|f^{\circ m}\left(Y_{t}\right)-f^{\circ m}\left(Y_{t}^{s,t}\right)\right|\nonumber \\
 & \leq & \left|f^{\circ m}\right|_{Lip\left(1\right)}\left|Y_{t}-Y_{t}^{s,t}\right|\nonumber \\
 & \leq & C_{p}\left|f^{\circ m}\right|_{Lip\left(1\right)}\left|f\right|_{Lip\left(\gamma\wedge\left(\lfloor p\rfloor+1\right)\right)}^{\gamma\wedge\left(\lfloor p\rfloor+1\right)}\left|X\right|_{p-var,\left[s,t\right]}^{\gamma\wedge\left(\lfloor p\rfloor+1\right)}\nonumber \\
 & \leq & C_{p}\left|f^{\circ m}\right|_{Lip\left(1\right)}\left|f\right|_{Lip\left(\gamma\wedge\left(\lfloor p\rfloor+1\right)\right)}^{\gamma\wedge\left(\lfloor p\rfloor+1\right)}\left(\left|X\right|_{p-var,\left[0,T\right]}\vee1\right)^{\lfloor p\rfloor+1}\left|X\right|_{p-var,\left[s,t\right]}^{\gamma-m}.\label{eq:first term}
\end{eqnarray}
where the crucial step is in the final line where we used $\lfloor\gamma\rfloor-m\leq\lfloor p\rfloor+1$. 

For the second term in (\ref{eq:approx by smooth paths}), 

\begin{eqnarray}
 &  & \left|f^{\circ m}\left(Y_{t}^{s,t}\right)-\sum_{k=0}^{\lfloor\gamma\rfloor-m}f^{\circ\left(m+k\right)}\left(Y_{s}\right)S_{k}\left(x^{s,t}\right)_{s,t}\right|\nonumber \\
 & = & \left|\int_{s<s_{1}<\ldots<s_{\lfloor\gamma\rfloor-m}<t}f^{\circ\lfloor\gamma\rfloor}\left(Y_{s_{1}}^{s,t}\right)-f^{\circ\lfloor\gamma\rfloor}\left(Y_{s}\right)\mathrm{d}x_{s_{1}}^{s,t}\ldots\mathrm{d}x_{s_{N-m}}^{s,t}\right|\\
 & \leq & \left|f^{\circ\lfloor\gamma\rfloor}\right|_{Lip\left(1\right)}\left|Y_{\cdot}^{s,t}\right|_{p-var,\left[s,t\right]}^{\gamma-\lfloor\gamma\rfloor}\left|X\right|_{p-var,\left[s,t\right]}^{\lfloor\gamma\rfloor-m}\nonumber \\
 & \leq & C_{p}\left|f^{\circ\lfloor\gamma\rfloor}\right|_{Lip\left(1\right)}\left(\left|f\right|_{Lip\left(\gamma\wedge\lfloor p\rfloor+1\right)}\vee1\right)^{p\left(\gamma-\lfloor\gamma\rfloor\right)}\label{eq:second term}\\
 &  & \times\left(\left|X\right|_{p-var,\left[0,T\right]}\vee1\right)^{\left(p-1\right)\left(\gamma-\lfloor\gamma\rfloor\right)}\left|X\right|_{p-var,\left[s,t\right]}^{\gamma-m}.
\end{eqnarray}

Combining (\ref{eq:approx by smooth paths}), (\ref{eq:first term})
and (\ref{eq:second term}), we have 
\begin{eqnarray*}
 &  & \left|f^{\circ\left(m\right)}\left(Y_{t}\right)-\sum_{k=0}^{\lfloor\gamma\rfloor-m}f^{\circ\left(m+k\right)}\left(Y_{s}\right)X_{s,t}^{k}\right|\\
 & \leq & 2C_{p}\max_{\lfloor\gamma\rfloor-\lfloor p\rfloor+1\leq m\leq\lfloor\gamma\rfloor}\left|f^{\circ m}\right|_{Lip\left(1\right)}\left(\left|f\right|_{Lip\left(\gamma\wedge\lfloor p\rfloor+1\right)}\vee1\right)^{\lfloor p\rfloor+1}\\
 &  & \times\left(\left|X\right|_{p-var,\left[0,T\right]}\vee1\right)^{\lfloor p\rfloor+1}\left|X\right|_{p-var,\left[s,t\right]}^{\lfloor\gamma\rfloor-m},
\end{eqnarray*}

It now suffices to show (\ref{eq:rpFTC}). 

Note that for $m\leq\lfloor\gamma\rfloor-\lfloor p\rfloor$, 
\begin{eqnarray*}
\int_{s}^{t}f^{\circ\left(m+1\right)}\left(Y_{u}\right)\mathrm{d}X_{u} & = & \int_{s}^{t}D\left(f^{\circ m}\right)f\left(Y_{u}\right)\mathrm{d}X_{u}\\
 & = & \int_{s}^{t}D\left(f^{\circ m}\right)\mathrm{d}Y_{u}\\
 & = & f^{\circ m}\left(Y_{t}\right)-f^{\circ m}\left(Y_{s}\right).
\end{eqnarray*}

\end{proof}


\begin{thebibliography}{1}
\bibitem[1]{HaraHino10}K. Hara and M. Hino. Fractional order Taylor\textquoteright{}s
series and the neo-classical inequality. Bull. Lond. Math. Soc. 42
467\textendash{}477, 2010.

\bibitem[2]{Boe15}H. Boedihardjo, Decay rate of iterated integrals
of branched rough paths, arXiv:1501.05641, 2015. 

\bibitem[3]{Dav07}A. Davie, Differential equations driven by rough
paths: an approach via discrete approximation, Appl. Math. Res. Express,
AMRX, (2):Art. ID abm009, 40, 2007.

\bibitem[4]{FV10}P. Friz, N. Victoir, \textit{Multidimensional Stochastic
Processes as Rough Paths. Theory and Applications}, Cambridge Studies
of Advanced Mathematics, Vol. 120, Cambridge University Press, 2010.

\bibitem[5]{FV08}P. Friz, N. Victoir, Euler estimates for rough differential
equations.\textit{ J. Differential Equations}, 244(2):388\textemdash{}412,
2008.

\bibitem[6]{Gub04}M. Gubinelli, Controlling Rough Paths, \textit{J.
Funct. Anal.}, 216:86-140, 2004. 

\bibitem[7]{Lyons94}T. Lyons, Differential equations driven by rough
signals (I): an extension of an inequality of L. C. Young. Mathematical
Research Letters 1, 451-464, 1994.

\bibitem[8]{Lyons98}T. Lyons, Differential equations driven by rough
signals, Rev. Mat. Iberoamericana., Vol. 14 (2), 215\textendash{}310,
1998.

\bibitem[9]{Young}L. C. Young. An inequality of Hölder type connected
with Stieltjes integration. \textit{Acta Math.}, (67):251\textendash{}282,
1936.\end{thebibliography}
\end{document}